\documentclass[a4paper, 11pt]{article}

\usepackage[T1]{fontenc}
\usepackage{amsmath}
\usepackage{amsthm}
\usepackage{amssymb}
\usepackage{ifpdf}
\usepackage{graphicx}
\usepackage{xcolor}
\usepackage{algorithm}
\usepackage{clrscode3e}
\usepackage[subnum]{cases}
\usepackage[hidelinks]{hyperref}

\usepackage{tabularx}
\usepackage{multirow}
\usepackage{array}
\usepackage{colortbl}
\usepackage{hhline}

\usepackage[format=hang]{caption}
\usepackage[labelformat=simple,format=hang]{subcaption}
\renewcommand\thesubfigure{(\alph{subfigure})}

\usepackage{tikz}
\usepackage{tkz-graph}
\usetikzlibrary{snakes}
\usetikzlibrary{arrows}

\usepackage{geometry}
\newtheorem{theorem}{Theorem}
\newtheorem{lemma}[theorem]{Lemma}

\newtheorem{definition}[theorem]{Definition}

\def \ham{{\mathcal{HAM}}} 

\newcommand{\doublecell}[2]{\begin{tabular}{@{}c@{}}#1\\#2\end{tabular}}

\begin{document}

\title{Hamiltonian Maker-Breaker games on small graphs}
\date{}

\author{
	Milo\v{s} Stojakovi\'{c}
		\footnote{Department of Mathematics and Informatics, Faculty of Sciences, University of Novi Sad, Serbia.
		Partly supported by Ministry of Education and Science, Republic of Serbia, and
		Provincial Secretariat for Science, Province of Vojvodina. {\tt \{milos.stojakovic, nikola.trkulja\}@dmi.uns.ac.rs}}
	\and
	Nikola Trkulja
		\footnotemark[1]
}

\maketitle

\begin{abstract}
We look at the unbiased Maker-Breaker Hamiltonicity game played on the edge set of a complete graph $K_n$, where Maker's goal is to claim a Hamiltonian cycle. First, we prove that, independent of who starts, Maker can win the game for $n = 8$ and $n = 9$. Then we use an inductive argument to show that, independent of who starts, Maker can win the game if and only if $n \geq 8$. This, in particular, resolves in the affirmative the long-standing conjecture of Papaioannou from~\cite{Papaioannou}.

We also study two standard positional games related to Hamiltonicity game. For Hamiltonian Path game, we show that Maker can claim a Hamiltonian path if and only if $n \geq 5$, independent of who starts. Next, we look at Fixed Hamiltonian Path game, where the goal of Maker is to claim a Hamiltonian path between two predetermined vertices. We prove that if Maker starts the game, he wins if and only if $n \geq 7$, and if Breaker starts, Maker wins if and only if $n \geq 8$. Using this result, we are able to improve the previously best upper bound on the smallest number of edges a graph on $n$ vertices can have, knowing that Maker can win the Maker-Breaker Hamiltonicity game played on its edges.
	
To resolve the outcomes of the mentioned games on small (finite) boards, we devise algorithms for efficiently searching game trees and then obtain our results with the help of a computer.
\end{abstract}

\section{Introduction}

A positional game is a hypergraph $(X,\mathcal{F})$, where $X$ is a finite set representing the board of the game, and $\mathcal{F} \subseteq 2^{X}$ is a family of sets that we call \emph{winning sets}. In Maker-Breaker positional games, the players are called Maker and Breaker. In the course of the game, Maker and Breaker alternately claim unclaimed elements of the board $X$, one element at a time, until all of the elements are claimed. Maker wins the game if he claims all elements of a winning set, while Breaker wins if he claims an element in every winning set. Each game can be observed in two variants, depending on which player is to play first. One of the main questions in the theory of positional games is the existance of a \emph{winning strategy} for one of the players, when both are playing optimally. The player that has a strategy to win the game is referred to as the winner of the game.

The positional games we intend to study are played on graphs, and in particular, on the edge set of a complete graph $K_n$. Our prime interest lies with Hamiltonicity game~$\ham_n$, where the winning sets are the edge sets of all Hamiltonian cycles in $K_n$. The game was first introduced by Chv\'atal and Erd\H{o}s in~\cite{ChvatalErdos}, and since then it has been one of the most studied positional games on graphs, see~\cite{PositionalGamesBook, krivelevich2011critical}. It was shown in~\cite{ChvatalErdos} that Maker has a winning strategy for all sufficiently large $n$. Papaioannou~\cite{Papaioannou} later proved that Maker wins the game for all $n \geq 600$, and at the same time conjectured that the smallest $n$ for which Maker can win is $8$. Hefetz and Stich~\cite{HefetzStich} further improved the upper bound by showing that Maker wins for all $n \geq 29$. We note that these statements hold under the assumption that Maker is the first to play.

In the present paper, we determine the outcome of Hamiltonicity game for every value of $n$, and for each of the players starting the game. In particular, this resolves the mentioned long-standing conjecture of Papaioannou in the affirmative. As the trivial cases $n \leq 3$ can be handled directly, from now on we assume $n \geq 4$.
\begin{theorem} \label{t:hcg}
In the Maker-Breaker~$\ham_n$ game on $E(K_n)$, Maker, as first or second player, wins if and only if $n \geq 8$.
\end{theorem}

Next, we look at two games where Maker's goal is to claim a Hamiltonian path. In Hamiltonian Path game~$\mathcal{HP}_n$, first introduced in~\cite{Papaioannou}, the winning sets are the edge sets of all Hamiltonian paths in $K_n$. We are able to show the following.
\begin{theorem} \label{t:hpg}
In the Maker-Breaker Hamiltonian Path game~$\mathcal{HP}_n$ on $E(K_n)$, Maker, as first or second player, wins if and only if $n \geq 5$.
\end{theorem}
This theorem is a strengthening of Papaioannou's result from~\cite{Papaioannou} where he proved that Maker, as first player, can win $\mathcal{HP}_n$ if and only if $n \geq 5$.
	
In Fixed Hamiltonian Path game~$\mathcal{FHP}_n$, the goal of Maker is to claim a Hamiltonian path between two fixed (predetermined) vertices, $u,v\in V(K_n)$. Even though in the literature this game does not draw as much interest as $\ham_n$ and $\mathcal{HP}_n$, it has appeared as an auxiliary game when studying some other games on graphs, see e.g.~\cite{FixedHamPathPaper}.
\begin{theorem} \label{t:fhpg}	
In the Maker-Breaker Fixed Hamiltonian Path game~$\mathcal{FHP}_n$ on $E(K_n)$, Maker, as first player, wins if and only if $n \geq 7$, and as second player he wins if and only if $n \geq 8$.
\end{theorem}
Note that in the game~$\mathcal{FHP}_n$ the edge between the fixed vertices $u$ and $v$ actually does not participate in any winning set, so right away we obtain the same result for the game played on $E(K_n) \setminus \{(u,v)\}$.

Here we also want to mention two other standard positional games, Connectivity and Perfect Matching, where Maker's goal is to claim, respectively, a spanning tree and a perfect matching. For both of these  games it is straightforward to determine the outcome for every size of the board. Maker wins the Connectivity game as first player for every $n$, and as second player if and only if $n \geq 4$. In Perfect Matching game, where $n$ is even, Maker can win as first or second when $n$ is at least $6$. On top of that, he wins also for $n=2$ when he plays first.

Let us note that even though answering the question of who wins a game on a small board (with $n$ fixed) is a finite problem, it still may have a greater scientific impact for several reasons. First of all, the approaches we use to resolve standard positional games when $n$ is large often do not apply for small $n$, and in that case in order to determine the outcome we need to develop new methods. As we saw, resolving the ``small cases'' is straightforward for some games, but not for all.

Also, standard positional games, like Hamiltonicity, Hamiltonian Path, Connectivity, etc.\ are often used as auxiliary games when studying other positional games on graphs. Sometimes these auxiliary games have boards of fixed size, and knowing their outcome is essential for completing the analysis. An example of that can be found in~\cite{Sparse}, where one of the problems tackled was to estimate the smallest number of edges $\hat{m}(n)$ a graph on $n$ vertices can have, knowing that Maker as first player can win the Maker-Breaker Hamiltonicity game played on its edges. It was proved in~\cite{Sparse} that $ 2.5n \leq \hat{m}(n) \leq 21n$, for all $n \geq 1600$. We show how to apply our results about Hamiltonicity game in Theorem~\ref{t:hcg} and Fixed Hamiltonian Path game in Theorem~\ref{t:fhpg} to improve this upper bound, eventually obtaining the following.
\begin{theorem} \label{t:app}	
For $n \geq 336$, we have $\hat{m}(n) \leq 4n$.
\end{theorem}

Positional games are combinatorial games (sequential two player games with perfect information and no randomness involved), and it is well known (see, e.g.,~\cite{AIBook}) that we can find out which of the players has a strategy to win by simply traversing the whole game tree of the game. But this fact alone is of limited practical use, knowing that already for games on relatively small boards the game trees are way too big (they are exponentially large in the size of the game board) to be completely traversed by a computer. In particular, if the board of the game is $E(K_n)$, its size is $\binom{n}{2}$, so there are $\binom{n}{2}!$ different game plays.

Often there is no need to search through the whole game tree, as some moves are ``analogue'' to the others, we may arrive to the same game position more than once, and on top of that some game positions are ``similar'' to the others. In Section~\ref{s:aa} we devise a sophisticated set of algorithms that formalizes and exploits these ``similarities'' as part of the optimization of the brute force search algorithm. This enables us to write a computer program that efficiently calculates the outcome of all three mentioned games for enough initial values of $n$ to inductively determine the outcome of the game for every $n$.

When implementing the algorithms we aimed at doing it in a generic way to make our code easily adaptable for other positional games on graphs. Even though some algorithms are tailored to fit the particularities of the games we analyzed, most of them are generally applicable for  determining the outcome of positional games on graphs with the help of a computer.  Having that in mind, our software can be seen as a general framework for computer based attacks on positional games.

\subsection{Preliminaries and organization}

	Our graph-theoretic notation is standard and follows the one from \cite{West}. In particular, we
	use the following. For a graph $G$, let $V(G)$ denote its set of vertices and $E(G)$ its set of
	edges, where $v(G) = |V(G)|$ and $e(G) = |E(G)|$. For a set $A \subseteq V(G)$, let $G[A]$ denote
	the subgraph of $G$ induced on the vertex set $A$ and let $E_G(A)$ denote the set of edges of $G$
	with both endpoints in $A$.

	Results presented in this paper rely on the following statement, see e.g.~\cite{PositionalGamesBook}.
	\begin{theorem} \label{t:aux}
	If Maker (or Breaker) has a winning strategy in some positional game as second player, then he also
	has a winning strategy if he starts the game.
	\end{theorem}

    If a Maker-Breaker game played on edges of a graph is in progress, Maker's (respectively, Breaker's) vertex degree of a vertex $v$, denoted by $d_M(v)$ (respectively, $d_B(v)$), is the number of edges incident to $v$ that are claimed by Maker (respectively, Breaker). Maker's edge degree of an edge $e = (u,v)$, denoted by $ed_M(e)$, is defined to be $ed_M(e) := d_M(u) + d_M(v)$, and Breaker's edge degree is $ed_B(e) := d_B(u) + d_B(v)$.

The rest of the paper is organized as follows. In Section~\ref{s:hcg} we prove Theorem~\ref{t:hcg}, in Section~\ref{s:hpg} we prove Theorem~\ref{t:hpg} and in Section~\ref{s:fhpg} we prove Theorem~\ref{t:fhpg}. Section~\ref{s:aa} contains the description of the algorithms and optimization methods used in our computer programs, as well as the results we obtain with their help. In Section \ref{s:app} we show how our results can be applied in solving other problems in theory of positional games, by proving Theorem~\ref{t:app}. Finally, in Section~\ref{s:fw} we give some open problems and ideas for future work.

\section{Hamiltonicity game}
\label{s:hcg}

The following lemma lies in the foundation of our proof of Theorem~\ref{t:hcg}.

\begin{lemma}	\label{l:hcg_bc}
Independent of who starts the game $\ham_n$, Breaker wins for $n \in \{ 4, 5, 6, 7 \}$, and Maker wins for $n = 8$ and $n = 9$.
\end{lemma}

We prove this result with the help of a computer. Description of the approach used is postponed to Section~\ref{s:aa}.

The previous result gives us a firm induction base for proving that Maker can win $\ham_n$ for all $n \geq 8$. We note that Papaioannou already proved the induction step of size $2$ for Maker as first player, which can be paired up with our Lemma~\ref{l:hcg_bc} to resolve the case when Maker is first. Inspired by his work, in the following lemma we prove the same induction step for Maker playing as second player, which eventually completes the proof of Theorem~\ref{t:hcg}.
	
\begin{lemma} \label{l:hcg_is}
If Maker, as second player, wins $\ham_n$, he can also win $\ham_{n+2}$ as second player.
\end{lemma}

\begin{proof}
Suppose that Maker, as second player, has a strategy to win $\ham_n$. We will present Maker's winning strategy for $\ham_{n+2}$.
		
At the beginning of the game, Breaker will claim some edge $(u,v)$ and Maker's response will be to claim an edge $m=(v,w)$ for some $w$. Let $V_1:= V(K_{n+2}) \setminus \{v,w\}$. We can now partition the set of all unclaimed edges into two sets, the first set being $E_1 = E(V_1)$ and the second one $E_2 = E(K_{n+2}) \setminus
		(E_1 \cup \{ m, (u,v) \})$. These edge sets are shown in the Figure \ref{f:board_with_two_partitions}.
		
		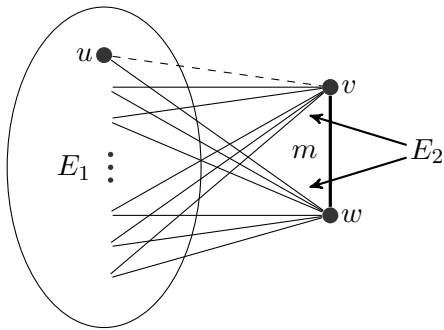
\begin{figure}[h]
			\begin{center}
				\begin{tikzpicture}
				[scale=0.85, every node/.style={circle, fill=black!80, minimum size=6pt, inner sep=0pt}]
				
				\node (n3) at (1.5,1) [label=right:$w$] {};
				\node (n4) at (1.5,3) [label=right:$v$] {};
				
				\node (n3s) at (1.05,1.4) [fill=none] {};
				\node (n4s) at (1.05,2.6) [fill=none] {};
				
				\node (n0) at (-2,0) [fill=none] {};
				\node (n1) at (-2,0.5) [fill=none] {};
				\node (n2) at (-2,1) [fill=none] {};
				
				\node (d1) at (-1.9,1.55) [minimum size=2pt] {};
				\node (d2) at (-1.9,1.75) [minimum size=2pt] {};
				\node (d3) at (-1.9,1.95) [minimum size=2pt] {};
				
				\node (n7) at (-2,2.5) [fill=none] {};
				\node (n8) at (-2,3) [fill=none] {};
				\node (n9) at (-2,3.5) [label=left:$u$] {};
				
				\node (E1) at (-2,1.75) [fill=none,label=left:$E_1$] {};
				\node (E2) at (3,2) [fill=none] {$E_2$};
				
				\foreach \from/\to in {n3/n0,n3/n1,n3/n2,n3/n7,n3/n8,n3/n9}
					\draw (\from) -- (\to);
				\foreach \from/\to in {n4/n7,n4/n8,n4/n0,n4/n1,n4/n2}
					\draw (\from) -- (\to);
				
				\draw [dashed] (n9) -- (n4) node {};
				\draw [very thick] (n3) -- (n4) node [midway,fill=none,label=left:$m$] {};
				\draw [->,>=stealth',thick] (E2) -- (n3s);
				\draw [->,>=stealth',thick] (E2) -- (n4s);
				\draw (-2,1.75) ellipse (1.5cm and 2.5cm);
				
				\end{tikzpicture}			
			\end{center}
			\caption{Game board, with the partition into two sets.}
			\label{f:board_with_two_partitions}
		\end{figure}

		From this point on, Maker will try to respond to every move of Breaker by claiming the edge in the same edge set as Breaker e.g.~whenever Breaker claims an edge in $E_i$ Maker will also try to claim an edge from $E_i$, $i= 1, 2$. If this is not possible, he will claim an arbitrary edge in the other set. It is generally true that such ``extra'' edges can be ``forgotten'', i.e.~if Maker can accomplish something in the other set without the extra edge, he will still manage to do it with the extra edge, see e.g.~\cite{PositionalGamesBook}.
		
		Maker will play in $E_1$ using his winning strategy for $\ham_n$, meaning that by the time all edges in $E_1$ are claimed he will fully occupy some $n$-cycle $H_n$ in $E_1$.
		
		When playing in $E_2$ Maker's goal is to make sure that by the time all the edges are claimed (in both $E_1$ and $E_2$) he claimed a Hamiltonian cycle. This Hamiltonian cycle will be an extension of $H_n$ to a cycle on two more vertices, now also containing the edge $m$. More precisely, in order to extend the cycle $H_n=(x_1, x_2, \ldots, x_n, x_1)$, Maker wants to claim edges $(x_i, y_1), (x_{i+1}, y_2) \in E_2$, such that the edge $(x_i,x_{i+1}) \in H_n$
		and  $(y_1,y_2) = m$. That way, Maker would claim a Hamiltonian cycle $H_{n+2}=(x_1, \ldots, x_i, y_1, y_2, x_{i+1}, \ldots, x_n, x_1)$, as shown in Figure~\ref{f:extension_of_h_n}.

 		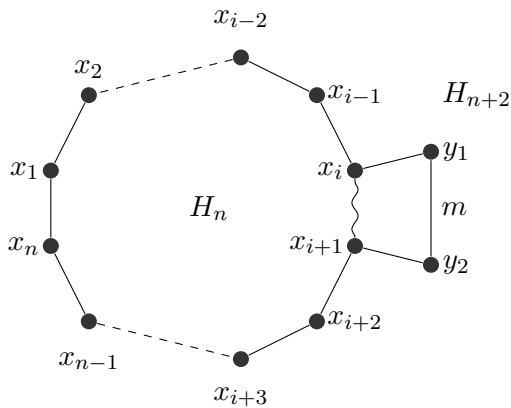
\begin{figure}[h]
			\begin{center}
				\begin{tikzpicture}
				[scale=1.0,every node/.style={circle, fill=black!80, minimum size=6pt, inner sep=0pt}]
				
				\node (n-1) at (-2.5,2) [label=left:$x_n$] {};
				\node (n0) at (-2,1) [label=below:$x_{n-1}$] {};
				\node (n1) at (0,0.5) [label=below:$x_{i+3}$] {};
				\node (n2) at (1,1) [label=right:$x_{i+2}$] {};
				\node (n3) at (1.5,2) [label=left:$x_{i+1}$] {};
				\node (n4) at (1.5,3) [label=left:$x_i$] {};
				\node (n5) at (1,4) [label=right:$x_{i-1}$] {};
				\node (n6) at (0,4.5) [label=above:$x_{i-2}$] {};
				\node (n7) at (-2,4) [label=above:$x_2$] {};
				\node (n8) at (-2.5,3) [label=left:$x_1$] {};
				
				\node (hn) at (0,2.5) [fill=none,label=left:{$H_n$}] {};
				\node (hn2) at (2.5,4) [fill=none,label=right:{$H_{n+2}$}] {};
				
				\node (y1) at (2.5,3.25) [label=right:$y_1$] {};
				\node (y2) at (2.5,1.75) [label=right:$y_2$] {};
				
				\foreach \from/\to in {n1/n2,n2/n3,n4/n5,n5/n6,n4/y1,n3/y2,n0/n-1,n-1/n8,n8/n7}
					\draw (\from) -- (\to);
					
				\draw (y1) -- (y2) node [midway,fill=none,label=right:$m$] {};
				\draw[snake=snake,segment amplitude=.4mm] (n3) -- (n4);
				\draw[dashed] (n0) -- (n1);
				\draw[dashed] (n6) -- (n7);
				
				\end{tikzpicture}			
			\end{center}
			\caption{Extension of $H_n$ to $H_{n+2}$.}
			\label{f:extension_of_h_n}
		\end{figure}	
		
During the game, if for a vertex $x \in V_1$ both edges $(x,v)$ and $(x,w)$ are claimed by Maker (Breaker), we say the vertex $x$ is \emph{isolated by Maker (Breaker)}. If one of the edges $(x,v)$ and $(x,w)$ is claimed by Maker (Breaker) and the other one is unclaimed, we say the the vertex $x$ is \emph{half isolated by Maker (Breaker)}. Finally, if both edges $(x,v)$ and $(x,w)$ are unclaimed, we say that the vertex $x$ is \emph{free}.
When all the edges are claimed, we observe the following two situations.
		
\textbf{Situation 1:} Breaker isolated at most one vertex, and Maker claimed an edge in $E_2$ incident to each of vertices $v$ and $w$.

\textbf{Situation 2:} Breaker isolated exactly two vertices, Maker has one isolated vertex $a$, and Maker claimed an edge from $E_2 \setminus \{(a,v), (a,w)\}$ incident to each of vertices $v$ and $w$.

First, let us prove that in each of the two above mentioned situations Maker created a Hamiltonian cycle and thus won the game. We can apply simple labeling procedure, assigning a label set to each vertex $x \in H_n$. This set contains $0$ if and only if edge $(x,v)$ is claimed by Maker and $1$ if and only if edge $(x,w)$ is claimed by Maker.

In Situation~1, since Maker claimed an edge in $E_2$ incident to each of vertices $v$ and $w$, at least one vertex is labeled with $0$, and at least one vertex is labeled with $1$. Also, there is at most one vertex with empty label set, so in every Hamiltonian cycle $H_n$ there must exist two adjacent vertices $x_i$ and $x_{i+1}$ containing labels 0 and 1, respectively. Maker will then use these vertices $x_i$ and $x_{i+1}$ to extend the cycle $H_n$.

In Situation~2 there is a vertex $a$ with label set $\{0,1\}$. If this vertex is adjacent to some vertex $x \in H_n$ whose label set is not empty, Maker can extend the cycle $H_n$ using $a$ and $x$. Otherwise, if vertex $a$ is adjacent to vertices whose label set is empty (two vertices isolated by Breaker), then Maker will be able to extend $H_n$ for the same reasons as in Situation~1, since he claimed an edge from $E_2 \setminus \{(a,v), (a,w)\}$ incident to each of vertices $v$ and $w$.
		
Next, we are going to give a strategy for Maker that will enable him to always end up in one of these two situations.

\paragraph{Maker's strategy on $E_2$.} In each move, Maker will apply the first applicable rule in the following list:

\begin{enumerate}

\item If there is only one free vertex $x \in V_1$ left, and Maker has not claimed any edges incident to a vertex  $y \in \{v, w\}$, he claims the edge $(x,y)$.

\item If there exists a vertex half isolated by Breaker, Maker picks an arbitrarily vertex $x$ half isolated by Breaker and claims the other edge incident to $x$.

\item If there exists a free vertex, Maker picks an arbitrarily free vertex and claims one of its free edges. Of the two edges, he prefers the edge incident with the vertex in $\{v, w\}$ incident with less edges claimed by Maker (if such vertex exists).

\item Maker claims a free edge incident to a vertex half isolated by Maker. Of all such edges, he prefers edges incident with the vertex in $\{v, w\}$ incident with less edges claimed by Maker (if such vertex exists), breaking ties arbitrarily.
\end{enumerate}

As any free edge is incident either with a vertex half isolated by Breaker, or a free vertex, or a vertex half isolated by Maker, this is a valid strategy.

Initially, Breaker has claimed one edge, the edge $(u,v)$. For his first move in $E_2$ he essentially has only three different choices, shown in Figure~\ref{f:options_for_breakers_second_move} (along with the move of Maker that will follow in each of the cases).
		
		\begin{figure}[h]
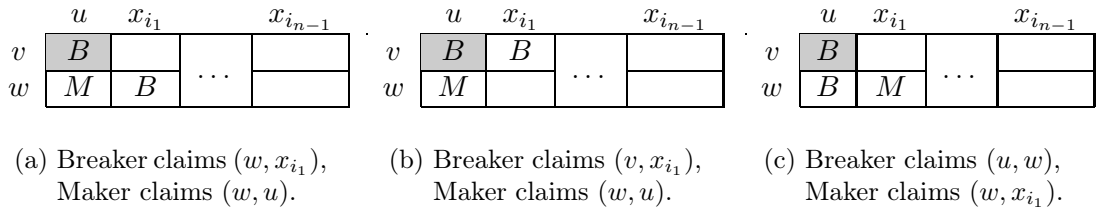

			\begin{center}
				\begin{subfigure}[b]{0.32\textwidth}
					\captionsetup{margin=0.34cm}				
					\begin{center}
					\begin{tabular}{|c|c|c|c|c|c|}
						\multicolumn{1}{c}{} & \multicolumn{1}{c}{$u$} & \multicolumn{1}{c}{$x_{i_1}$} &
						\multicolumn{2}{c}{} & \multicolumn{1}{c}{$x_{i_{n-1}}$} \\
						\hhline{|~|-|-|-|-|-|}
						\multicolumn{1}{c|}{$v$} & \cellcolor[gray]{0.8}$B$ &  &
						\multicolumn{2}{c|}{\multirow{2}*{\ldots}} &  \\
						\cline{2-3}\cline{6-6}
						\multicolumn{1}{c|}{$w$} & $M$ & $B$ & \multicolumn{2}{c|}{} &  \\
						\cline{2-6}
					\end{tabular}
					\end{center}	
				\caption{Breaker claims $(w,x_{i_1})$, \\ Maker claims $(w,u)$.}
				\label{f:breakers_second_move_1}
		        \end{subfigure}
				\begin{subfigure}[b]{0.32\textwidth}
					\captionsetup{margin=0.34cm}				
					\begin{center}
					\begin{tabular}{|c|c|c|c|c|c|}
						\multicolumn{1}{c}{} & \multicolumn{1}{c}{$u$} & \multicolumn{1}{c}{$x_{i_1}$} &
						\multicolumn{2}{c}{} & \multicolumn{1}{c}{$x_{i_{n-1}}$} \\
						\hhline{|~|-|-|-|-|-|}
						\multicolumn{1}{c|}{$v$} & \cellcolor[gray]{0.8}$B$ & $B$ &
						\multicolumn{2}{c|}{\multirow{2}*{\ldots}} &  \\
						\cline{2-3}\cline{6-6}
						\multicolumn{1}{c|}{$w$} & $M$ &  & \multicolumn{2}{c|}{} &  \\
						\cline{2-6}
					\end{tabular}
					\end{center}	
				\caption{Breaker claims $(v,x_{i_1})$, \\ Maker claims $(w,u)$.}
				\label{f:breakers_second_move_2}
		        \end{subfigure}
				\begin{subfigure}[b]{0.32\textwidth}
					\captionsetup{margin=0.34cm}				
					\begin{center}
					\begin{tabular}{|c|c|c|c|c|c|}
						\multicolumn{1}{c}{} & \multicolumn{1}{c}{$u$} & \multicolumn{1}{c}{$x_{i_1}$} &
						\multicolumn{2}{c}{} & \multicolumn{1}{c}{$x_{i_{n-1}}$} \\
						\hhline{|~|-|-|-|-|-|}
						\multicolumn{1}{c|}{$v$} & \cellcolor[gray]{0.8}$B$ &  &
						\multicolumn{2}{c|}{\multirow{2}*{\ldots}} &  \\
						\cline{2-3}\cline{6-6}
						\multicolumn{1}{c|}{$w$} & $B$ & $M$ & \multicolumn{2}{c|}{} &  \\
						\cline{2-6}
					\end{tabular}
					\end{center}	
				\caption{Breaker claims $(u,w)$, \\ Maker claims $(w,x_{i_1})$.}
				\label{f:breakers_second_move_3}
		        \end{subfigure}
			\end{center}
			\caption{Three options for Breaker's first move in $E_2$ followed by Maker's response, shown
			in the form of incidence matrix.}
			\label{f:options_for_breakers_second_move}
		\end{figure}
		
	Now, let us analyze each of the three cases from Figure \ref{f:options_for_breakers_second_move} individually.

		\textbf{Case (a):} In this case, Breaker first has an option to isolate a vertex $x_{i_1}$. If he indeed does so, Maker will apply Rule 3 and claim an edge incident to a vertex $v$, otherwise Maker will apply Rule 2 and claim an edge incident to $v$. When Breaker isolates a vertex, Maker will apply Rule 2 in the remainder of the game to prevent Breaker from isolating more vertices. The proposed strategy thus inevitably puts Maker in Situation 1.
		
		\textbf{Case (b):} Again, Breaker can start by isolating a vertex $x_{i_1}$ which inevitably puts Maker in Situation 1 as this is analogous to Case (a). Otherwise, Maker will wait (by applying Rule 2) for Breaker to isolate some vertex, until there is only one free vertex left. If Breaker isolates a vertex during this period of the game Maker will apply Rule 3 claiming an edge incident to $v$ and he will spend the remainder of the game applying Rule 2 in order to prevent Breaker from isolating more vertices. It can happen that Breaker does not isolate any vertex by the time there is only one free vertex left, in which case Maker will end the game by applying the Rule 1 as shown in Figure~\ref{f:breaker_abusing_maker_responds}. Therefore, Maker will inevitably find himself in Situation 1.
		
		\begin{figure}[h]
			\begin{center}
				\begin{tabular}{|c|c|c|c|c|c|c|c|c|c|}
					\multicolumn{1}{c}{} & \multicolumn{1}{c}{$u$} & \multicolumn{1}{c}{$x_{i_1}$} &
					\multicolumn{1}{c}{$x_{i_2}$} & \multicolumn{2}{c}{} & \multicolumn{1}{c}{$x_{i_{n-4}}$} &
					\multicolumn{1}{c}{$x_{i_{n-3}}$} & \multicolumn{1}{c}{$x_{i_{n-2}}$} &
					\multicolumn{1}{c}{$x_{i_{n-1}}$} \\
					\hhline{|~|-|-|-|-|-|-|-|-|-|}
					\multicolumn{1}{c|}{$v$} & \cellcolor[gray]{0.8}$B$ & \cellcolor[gray]{0.8}$B$ &
					\cellcolor[gray]{0.8}$B$ & \multicolumn{2}{c|}{\multirow{2}*{\ldots}} &
					\cellcolor[gray]{0.8}$B$ & \cellcolor[gray]{0.8}$B$ & $B$ & $M$ \\
					\hhline{|~|-|-|-|~|~|-|-|-|-|}
					\multicolumn{1}{c|}{$w$} & \cellcolor[gray]{0.8}$M$ & \cellcolor[gray]{0.8}$M$ &
					\cellcolor[gray]{0.8}$M$ & \multicolumn{2}{c|}{} & \cellcolor[gray]{0.8}$M$ &  &  & \\
					\cline{2-10}
				\end{tabular}
			\end{center}
			\caption{Breaker can play so that Rule~1 is applied}
			\label{f:breaker_abusing_maker_responds}
		\end{figure}
		
		\textbf{Case (c):} In this case we will need to further analyze each of the three possible Breaker's moves. Breaker can either claim an edge incident to the only vertex half isolated by Maker, or for some free vertex
        $x \in V_1$ he can claim $(x,w)$ or $(x,v)$. These three possibilities are shown in Figure~\ref{f:third_case_analysis} (along with the move of Maker that will follow in each of the cases).
		
		\begin{figure}[h]
			\renewcommand\thesubfigure{(\roman{subfigure})}
			\begin{center}
				\begin{subfigure}[b]{0.32\textwidth}
					\captionsetup{margin=0.33cm}
					\begin{center}
					\begin{tabular}{|c|c|c|c|c|c|}
						\multicolumn{1}{c}{} & \multicolumn{1}{c}{$u$} & \multicolumn{1}{c}{$x_{i_1}$} &
						\multicolumn{1}{c}{$x_{i_2}$} & \multicolumn{2}{c}{} \\
						\hhline{|~|-|-|-|-|-|}
						\multicolumn{1}{c|}{$v$} & \cellcolor[gray]{0.8}$B$ & $B$ & $M$ &
						\multicolumn{2}{c|}{\multirow{2}*{\ldots}} \\
						\hhline{|~|-|-|-|~|~|}
						\multicolumn{1}{c|}{$w$} & \cellcolor[gray]{0.8}$B$ & \cellcolor[gray]{0.8}$M$ &  &
						\multicolumn{2}{c|}{} \\
						\cline{2-6}
					\end{tabular}
					\end{center}
				\caption{Breaker claims $(v,x_{i_1})$, \\ Maker claims $(v,x_{i_2})$.}
				\label{f:breakers_third_move_1}
		        \end{subfigure}
				\begin{subfigure}[b]{0.32\textwidth}
					\captionsetup{margin=0.33cm}				
					\begin{center}
					\begin{tabular}{|c|c|c|c|c|c|}
						\multicolumn{1}{c}{} & \multicolumn{1}{c}{$u$} & \multicolumn{1}{c}{$x_{i_1}$} &
						\multicolumn{1}{c}{$x_{i_2}$} & \multicolumn{2}{c}{} \\
						\hhline{|~|-|-|-|-|-|}
						\multicolumn{1}{c|}{$v$} & \cellcolor[gray]{0.8}$B$ &  & $M$ &
						\multicolumn{2}{c|}{\multirow{2}*{\ldots}} \\
						\hhline{|~|-|-|-|~|~|}
						\multicolumn{1}{c|}{$w$} & \cellcolor[gray]{0.8}$B$ & \cellcolor[gray]{0.8}$M$ & $B$ &
						\multicolumn{2}{c|}{} \\
						\cline{2-6}
					\end{tabular}
					\end{center}	
				\caption{Breaker claims $(w,x_{i_2})$, \\ Maker claims $(v,x_{i_2})$.}
				\label{f:breakers_third_move_2}
		        \end{subfigure}
				\begin{subfigure}[b]{0.33\textwidth}
					\captionsetup{margin=0.33cm}			
					\begin{center}
					\begin{tabular}{|c|c|c|c|c|c|}
						\multicolumn{1}{c}{} & \multicolumn{1}{c}{$u$} & \multicolumn{1}{c}{$x_{i_1}$} &
						\multicolumn{1}{c}{$x_{i_2}$} & \multicolumn{2}{c}{} \\
						\hhline{|~|-|-|-|-|-|}
						\multicolumn{1}{c|}{$v$} & \cellcolor[gray]{0.8}$B$ &  & $B$ &
						\multicolumn{2}{c|}{\multirow{2}*{\ldots}} \\
						\hhline{|~|-|-|-|~|~|}
						\multicolumn{1}{c|}{$w$} & \cellcolor[gray]{0.8}$B$ & \cellcolor[gray]{0.8}$M$ & $M$ &
						\multicolumn{2}{c|}{} \\
						\cline{2-6}
					\end{tabular}
					\end{center}	
				\caption{Breaker claims $(v,x_{i_2})$, \\ Maker claims $(w,x_{i_2})$.}
				\label{f:breakers_third_move_3}
		        \end{subfigure}
			\end{center}
			\caption{Further analysis of Case (c) from Figure \ref{f:options_for_breakers_second_move}.}
			\label{f:third_case_analysis}
		\end{figure}

		In cases~(i) and~(ii), which are essentially the same, Maker claimed an edge incident to each of vertices $v$ and $w$ while Breaker already isolated a vertex $u$. Maker will continue the game with only goal to prevent Breaker from isolating more vertices, by applying Rule~2 whenever it is possible, and hence the game will eventually end up in Situation~1.

		In the third case we will first analyse the game until there is only one free vertex left. Maker is again focused on preventing the Breaker from isolating a vertex by applying Rule~2 whenever possible. If during this period of the game Maker claims some edge incident to $v$, Rule~1 will not be applied, and the game will finish in Situation~1. If it happens that Rule~1 needs to be applied, Maker will claim an edge $(v,x_{i_{n-1}})$, where $x_{i_{n-1}}$ is the last free vertex. Depending if Breaker now isolates a vertex $x_{i_{n-2}}$ or not, the game can end up in different situations. If Breaker does not isolate $x_{i_{n-2}}$ the game ends in Situation~1. Otherwise, after vertex $x_{i_{n-2}}$ is isolated by Breaker, Maker isolates $x_{i_1}$ and the game ends in Situation~2.
		
		Hence, Maker can always create one of the two winning situations.
	\end{proof}
	
	\newpage
	
	\begin{proof}[Proof of Theorem \ref{t:hcg}]
		Applying mathematical induction, with Lemma \ref{l:hcg_bc} as the base case and Lemma \ref{l:hcg_is}
		as the induction step we get that Maker, as second player, wins $\ham_n$ for all $n \geq 8$. Combining this 
		result with the result from Theorem~\ref{t:aux} we complete the proof of Theorem~\ref{t:hcg}.
	\end{proof}

\section{Hamiltonian path game}
\label{s:hpg}

In order to complete the proof of Theorem~\ref{t:hpg} we used our computer program to obtain following result.
	
	\begin{lemma}
	\label{l:hpg_bc}
Independent of who starts the game $\mathcal{HP}_n$, Breaker wins for $n = 4$, and Maker wins for $n \in \{5,6,7\}$.
	\end{lemma}
Description of the proof and the approach used is postponed to Section~\ref{s:aa}.

With the help of the previous lemma and Theorem~\ref{t:hcg}, the proof of Theorem~\ref{t:hpg} is straightforward.
	
	\begin{proof}[Proof of Theorem~\ref{t:hpg}]
		It is easy to see that Maker's win in Hamiltonicity game implies Maker's win in Hamiltonian
		path game since each Hamiltonian cycle contains a Hamiltonian path.
		This means that Theorem~\ref{t:hcg} also tells us that Maker can win $\mathcal{HP}_n$ for all $n
		\geq 8$ as second player. When we combine this with the result from Lemma~\ref{l:hpg_bc} proof of
		Theorem~\ref{t:hpg} is complete.
	\end{proof}

\section{Fixed Hamiltonian Path game}
\label{s:fhpg}

The following lemma will cover several finite cases in the proof of Theorem~\ref{t:fhpg}.
	\begin{lemma}
	\label{l:fhcg_ib}
When Maker starts the game $\mathcal{FHP}_n$, Breaker wins for $n\in\{4, 5, 6\}$, and Maker wins for $n \in \{7,8,9\}$.

When Breaker is the one who starts, he wins for $n\in \{4, 5, 6, 7\}$, while Maker wins for $n \in \{8,9\}$.
	\end{lemma}
We prove this result with the help of a computer. Description of the approach used is postponed to the following section.
	
It remains to determine the outcome of the game $\mathcal{FHP}_n$ for larger values of $n$. We will not perform an induction step, but rather rely on our knowledge of the outcome of $\ham_n$.
Combining the following result with Theorem~\ref{t:hcg}, along with Lemma~\ref{l:fhcg_ib} that covers the few initial cases, we readily prove Theorem~\ref{t:fhpg}.
	
	\begin{lemma}
	\label{l:fxcg_is}
		If Maker, as second player, wins $\ham_n$, he can also win $\mathcal{FHP}_{n+2}$ as second player.
	\end{lemma}
	\begin{proof}
Suppose that Maker, as second player, has a strategy to win $\ham_n$. We will present Maker's winning strategy for $\mathcal{FHP}_{n+2}$, with two fixed vertices $u$ and $v$.
	
As we already noted, the edge $e = (u,v)$ is not interesting to either of the players since it is not contained in any of the winning sets. Let $V_1:= V(K_{n+2}) \setminus \{u,v\}$. We can partition the set of all unclaimed edges into two sets, the first set being $E_1 = E(V_1)$ and the second one $E_2 = E(K_{n+2}) \setminus (E_1 \cup \{ e \})$.

Now the goal of Maker is exactly the same as in the proof of Lemma~\ref{l:hcg_is} -- he wants to build an $n$-cycle on $E_1$, and he wants to play on $E_2$ to connect vertices $u$ and $v$ to two neighboring vertices of that $n$-cycle. As the structure of $E_1$ and $E_2$ here is exactly the same, with exception of $E_2$ having one more unclaimed edge (which is only beneficial for Maker), he can follow the same strategy as in the proof of Lemma~\ref{l:hcg_is} and win the game.
	\end{proof}

\section{Algorithmic analysis}
\label{s:aa}
	
	In this section we present the techniques and approaches used in our computer program to solve games
	$\ham_n$, $\mathcal{HP}_n$ and $\mathcal{FHP}_n$ when game boards are relatively small, as
	well as some of the major challenges we encountered during the process.
	
	Our program is available at \url{http://people.dmi.uns.ac.rs/~nikola.trkulja/ham.html}, where
	one can find all the information on how to run and use it, along with the source code.

\subsection{Traversing the game tree}	
	
	As we already noted, in order to determine the winner of a particular game, it is enough to completely
	traverse the game tree, which essentially comes down to simulation of all possible plays. This means that in each round we need to explore all valid moves
	a player can play. To achieve this we used the following algorithm which represents an adaptation of the standard and well-known game theoretic algorithm -- Minimax \cite{AIBook}.
	
	\begin{algorithm}
		\begin{codebox}
		
		\Procname{$\proc{Play}(G,P)$}
		\li $winner \gets \const{nil}$
		\li \For $i \gets 1$ \To $\attrib{G}{freeEdges.size}$ \label{al:move_ordering}
		\li		\Do $\proc{DoMove}(\attrib{G}{freeEdges[i]},G,P)$
		\li			\If $\proc{PlayerWins}(G,P)$
		\li 			\Then
							$winner \gets P$
		\li				\Else
		\li					$winner \gets \proc{Play}(G,\attrib{P}{otherPlayer})$
						\End
		\li			$\proc{RevertMove}(\attrib{G}{freeEdges[i]},G,P)$
		\li			\If $winner \isequal P$
		\li				\Then \Return $P$
						\End
				\End				
		\li \Return $\attrib{P}{otherPlayer}$
		
		\end{codebox}
		\caption{}
		\label{a:basic}
	\end{algorithm}
	
	Main procedure of our algorithm is procedure $\proc{Play}$, in charge of
	recursively exploring the whole game subtree starting from the given tree node (containing the current game board $G$,
	and the identity of the player $P$ whose turn it is). The value returned by procedure $\proc{Play}$
    corresponds to the player winning the game on board $G$, when player $P$ moves first.
    As the name suggests, procedures
	$\proc{DoMove}(e,G,P)$ and $\proc{RevertMove}(e,G,P)$ are used to mark and unmark, respectively, the edge $e$ as
	claimed by player $P$ on board $G$.

    The winner detection is an essential part of the algorithm, done by procedure $\proc{PlayerWins}$. Every time the winner is not
	immediately detected by procedure $\proc{PlayerWins}$, we continue recursively exploring the
	game tree. The algorithm stops processing the remaining
	possible moves for current player $P$ as soon there is a move which guaranties the win. On the other hand, if it turns out that there is no such move, we know that the
    other player wins the game.

\subsection{Isomorphism pruning}
	
	Basic version of the algorithm we just described can be improved so that the same game
	subtrees are not explored multiple times. This is where \emph{transposition table} \cite{AIGames}
	jumps into picture. Straightforward application of transposition table would be to store the winner for
	each game tree node, and then use these values to skip processing identical nodes. But the exponential
	size of the game tree renders such solution impossible because it would require huge amounts of
	computer memory. Much more efficient way of handling this situation is to recognize the isomorphic nodes
    of the game tree. For this task we implemented Brendan McKay's \emph{canonical labeling} algorithm
	\cite{McKayPGI}. Canonical labeling of a graph is defined in the following way.
	
	\begin{definition}
	Canonical labeling is a function $C$ mapping graphs to graphs, such that every graph $G_1$ is isomorphic to $C(G_1)$, and for every graph $G_2$ isomorphic with $G_1$ we have $C(G_2)=C(G_1)$.
	\end{definition}
	
	In other words, canonical labeling gives a unique representative for each class of isomorphic
	graphs, and with McKay's algorithm we can define and compute this graph in an efficient manner. That enables us to prune the game tree significantly, as we can store the winner only for canonical labelings.

\subsubsection{Edge colored graphs}

	When a Maker-Breaker game is played on the edge set of a graph, the situation on the game board can be seen as a 3-coloring of the
    edges, as each edge is either unclaimed (black), claimed by Maker (red), or claimed by Breaker (blue).
	
	One of the issues in application of McKay's canonical labeling algorithm emerges
	right here because the original version only handles non-colored and vertex colored graphs
	out of the box. In~\cite{McKayNauty} the author of the algorithm described how edge colored graph can
	be transformed into equivalent vertex colored graph which can then be fed as an input to canonical
	labeling algorithm (an interested reader can find more details in~\cite{Spermann08}). Unfortunately
	this transformation produces a graph with more vertices than the original graph (even when we have just two
	colors, the number of vertices doubles). Because of this, we did the implementation from scratch, 
	having~\cite{Turner07} in mind, where authors adapted the canonical labelling algorithm to work with edge
	colored graphs by distinguishing the vertices according to the colors of incident edges.

\subsubsection{Non-transitive winning sets}
	
	Another issue with detection of isomorphic game tree nodes emerges when dealing with Fixed Hamiltonian
	path game since this game has an additional property compared to the other two games we worked on -- there are two
    predefined vertices that require special attention. Therefore, we have to incorporate an additional condition,
    only taking into account the isomorphisms which map those two vertices to themselves.

    We further slice down each isomorphism class into subclasses defined in the following way.

	\begin{definition}
	We say that graph $G_1$ with predefined vertices $u_1, v_1 \in V(G_1)$ is in the same isomorphism
	subclass as graph $G_2$ with predefined vertices $u_2, v_2 \in V(G_2)$ if and only if $C(G_1) =
	C(G_2)$, and isomorphism functions $f_1$ and $f_2$ transforming $G_1$ and $G_2$	into $C(G_1)$, respectively, are such that $\{ f_1(u_1), f_1(v_1) \} = \{ f_2(u_2), f_2(v_2) \}$.
	\end{definition}
	
	Applying this definition to our problem, we chose sets $\{ u_1, v_1 \}$ and $\{ u_2, v_2 \}$ to be
	sets of endpoints of winning sets in $\mathcal{FHP}_n$, meaning that Maker wins the game on	$G_1$ if and only if he wins on $G_2$.

\subsection{Move ordering}

	Our algorithm is not processing the whole game tree but rather skipping some of the subtrees
	whenever it is sure about the outcome of that particular subtree. For example, when Maker is the one
	to claim an edge, we will test all the possible moves that he can play, one after another, until we find the one that guarantees his win (provided that such a move exists).
    Of course, it would be better that we probe a winning move as early as possible. Having this in mind we tested a number of heuristics for defining
	the order of the iteration over free edges in line~\ref{al:move_ordering} of Algorithm~\ref{a:basic}.
	
	We relied on sorting the free edges according to a predefined criterion. Particularly, free edges are sorted in non-decreasing order according to Maker's edge degree $ed_M$ for Maker's moves, and in non-increasing order according to Breaker's edge degree $ed_B$ for Breaker's moves.
    The idea behind this is that, informally speaking, Maker may want to prefer playing where his ``weakest link'' is, while Breaker may prefer choosing to enforce his ``strongest disconnecting point'' first.
	
	It turned out this optimization technique, implemented as described, is one of the most
	important ones, and it had a tremendous impact on the performance of our algorithm making the number of
	traversed game tree nodes significantly smaller.

\subsection{Winner detection}

	Procedure $\proc{PlayerWins}(G,P)$ is essential for checking if	the board of the game $G$ is such that the player $P$ won the game\footnote{If $P$ is Maker, he won if he claimed a whole winning set, and if $P$ is Breaker, he won if he claimed an edge in each of the winning sets.}. In order to efficiently
	implement this task we always did pre-computation of the set $\mathcal{F}$ of winning sets that are free of Breaker's edges. E.g.~in
	game $\ham_n$ we pre-generate the set $\mathcal{F}$ of all Hamiltonian cycles in $K_n$, and then we constantly
	update $\mathcal{F}$ so that in each moment it contains only the winning sets without any edges claimed by Breaker.

\subsection{Optimization of memory usage}	
	
	Several optimization techniques applied in our software are on the technical i.e.~implementation
    level, and majority of them are computer memory related. This is natural since we
	rely on transposition table in order to keep track of calculated values for visited game
	tree nodes, or more precisely their isomorphism classes, and even for relatively small graphs the number of
	these classes is extremely large.
	
	Utilizing sets in many components of our program we are able to achieve optimal ratio of time
	and memory efficiency. Sets are used to implement graphs, canonical labelings, winning
	sets and most importantly transposition tables. Of course, we needed efficient implementations of
	sets and for that job we used bitsets and high-performance 3rd party set implementations (in the form
	of Trove library~\cite{Trove}).
	
	After all this effort, our transposition tables were still not able to cope with the required amount
	of entries because of the limited computer memory. To solve this problem, we applied various heuristics
	to clean the tables i.e.~to remove the ``less important'' entries and free up the memory for the ``more
	important'' ones. One efficient technique to do this was to focus mainly on entries coming from lower
	depths of the game tree, say depths less than or equal to a fixed integer $d$, so whenever we detect that the 
	table is full all entries coming from depths larger than $d$ are removed from the table.

\subsection{Implementation, testing and results}
	
	Implementation was done using Java programming language, relying only on standard Java libraries
	with just one additional component, the already mentioned 3rd party high performance collections library -- Trove~\cite{Trove}.
    There are two main parts of our computer program, Minimax algorithm and Brendan McKay's canonical labeling
    algorithm. We chose to implement canonical labeling algorithm
	on our own in order to be more flexible and tailor the algorithm fully to our needs. We note that there are implementations of related algorithms available online (e.g.~\cite{Nauty}).
	
	While implementing the algorithm we followed the idea of doing it in a generic way
	to make our code easily adaptable for solving other positional games on graphs as well. Examples of this can be seen when browsing through
	our code since for all three games we exploited this approach, using the same code base with just
	minor modifications for each particular game.
	
	For testing purposes we used various revisions of Java version 6 and 7, and on top of that we ran the
	program on a number of different versions of three different operating systems,
	Windows, Mac OSX and Linux, all in the effort to minimize the possibility of error coming from any
	of these layers. We further conducted thorough testing of both of the main parts of our
	implementation using unit tests, in order to eliminate the possibility of errors coming from smaller
	components of our program. In the end, we put a lot of additional effort to verify that
	our implementation of Brendan McKay's algorithm is correct by comparing output of our
	implementation with the one provided by the author~\cite{Nauty} for many different types of graphs.
	
	The machine used for testing was Fujitsu Celsius M730 with Intel Xeon E5-1620 processor running at
    3.7 GHz, 32GB of RAM and with Ubuntu operating system. We were able to prove following statements using this machine.
	
	\begin{proof}[Proof of Lemma \ref{l:hcg_bc}]
		$ $
		\begin{center}
			\renewcommand{\arraystretch}{1.15}
			\begin{tabular}{|c|c|c|c|c|c|c|}
			\cline{2-7}
 			\multicolumn{1}{c|}{} & $\ham_4$ & $\ham_5$
 								  & $\ham_6$ & $\ham_7$
 								  & $\ham_8$ & $\ham_9$ \\
			\hline
			\doublecell{Running time}{Maker as first player} & 0s & 0s & 0.1s & 1s & 2s & 30s \\
			\hline
			Winner & Breaker & Breaker & Breaker & Breaker & Maker & Maker \\
			\hline
			\end{tabular}
			\renewcommand{\arraystretch}{1.0}
		\end{center}
		\begin{center}
			\renewcommand{\arraystretch}{1.15}
			\begin{tabular}{|c|c|c|c|c|c|c|}
			\cline{2-7}
 			\multicolumn{1}{c|}{} & $\ham_4$ & $\ham_5$
 								  & $\ham_6$ & $\ham_7$
 								  & $\ham_8$ & $\ham_9$ \\
			\hline
			\doublecell{Running time}{Maker as second player} & 0s & 0s & 0.1s & 1s & 3s & 68s \\
			\hline
			Winner & Breaker & Breaker & Breaker & Breaker & Maker & Maker \\
			\hline
			\end{tabular}
			\renewcommand{\arraystretch}{1.0}
		\end{center}		
	\end{proof}

	\begin{proof}[Proof of Lemma \ref{l:hpg_bc}]
		$ $
		\begin{center}
			\renewcommand{\arraystretch}{1.15}
			\begin{tabular}{|c|c|c|c|c|}
			\cline{2-5}
 			\multicolumn{1}{c|}{} & $\mathcal{HP}_4$ & $\mathcal{HP}_5$
 								  & $\mathcal{HP}_6$ & $\mathcal{HP}_7$ \\
			\hline
			\doublecell{Running time}{Maker as first player} & 0s & 0s & 0s & 0.3s \\
			\hline
			Winner & Breaker & Maker & Maker & Maker \\
			\hline
			\end{tabular}
			\renewcommand{\arraystretch}{1.0}
		\end{center}
		\begin{center}
			\renewcommand{\arraystretch}{1.15}
			\begin{tabular}{|c|c|c|c|c|}
			\cline{2-5}
 			\multicolumn{1}{c|}{} & $\mathcal{HP}_4$ & $\mathcal{HP}_5$
 								  & $\mathcal{HP}_6$ & $\mathcal{HP}_7$ \\
			\hline
			\doublecell{Running time}{Maker as second player} & 0s & 0s & 0.1s & 0.3s \\
			\hline
			Winner & Breaker & Maker & Maker & Maker \\
			\hline
			\end{tabular}
			\renewcommand{\arraystretch}{1.0}
		\end{center}
	\end{proof}

	\begin{proof}[Proof of Lemma \ref{l:fhcg_ib}]
		$ $
		\begin{center}
			\renewcommand{\arraystretch}{1.15}
			\begin{tabular}{|c|c|c|c|c|c|c|}
			\cline{2-7}
 			\multicolumn{1}{c|}{} & $\mathcal{FHP}_4$ & $\mathcal{FHP}_5$
 								  & $\mathcal{FHP}_6$ & $\mathcal{FHP}_7$
 								  & $\mathcal{FHP}_8$ & $\mathcal{FHP}_9$ \\
			\hline
			\doublecell{Running time}{Maker as first player} & 0s & 0s & 0.5s & 3s & 110s & 3990s \\
			\hline
			Winner & Breaker & Breaker & Breaker & Maker & Maker & Maker \\
			\hline
			\end{tabular}
			\renewcommand{\arraystretch}{1.0}
		\end{center}
		\begin{center}
			\renewcommand{\arraystretch}{1.15}
			\begin{tabular}{|c|c|c|c|c|c|c|}
			\cline{2-7}
 			\multicolumn{1}{c|}{} & $\mathcal{FHP}_4$ & $\mathcal{FHP}_5$
 								  & $\mathcal{FHP}_6$ & $\mathcal{FHP}_7$
 								  & $\mathcal{FHP}_8$ & $\mathcal{FHP}_9$ \\
			\hline
			\doublecell{Running time}{Maker as second player} & 0s & 0s & 0.1s & 2s & 230s & 16286s \\
			\hline
			Winner & Breaker & Breaker & Breaker & Breaker & Maker & Maker \\
			\hline
			\end{tabular}
			\renewcommand{\arraystretch}{1.0}
		\end{center}
	\end{proof}

\section{An application}
\label{s:app}

Generally speaking, the most direct way to show that a player wins a positional game is to exhibit an explicit strategy for the player, proving that it is a winning one. There are numerous cases (see~\cite{PositionalGamesBook} for examples) that such a strategy relies on strategies for one or more auxiliary positional games, where the auxiliary games are often confined to some part of the board and/or some stage of the game play.

Our results are applicable when such auxiliary games have Hamiltonian cycles or Hamiltonian paths as winning sets, and they are played on boards of fixed size. An example of that can be found in~\cite[Theorem 1.4]{Sparse}, when estimating the smallest number of edges $\hat{m}(n)$ a graph on $n$ vertices can have, knowing that Maker as the first player can win the Maker-Breaker Hamiltonicity game played on its edges. Using our results from Theorem~\ref{t:hcg} and Theorem~\ref{t:fhpg}, we are able to improve the upper bound on $\hat{m}(n)$.

In~\cite{Sparse} it was showed that $\hat{m}(n) \leq 21n$ for $n \geq 1600$, by explicit construction of a graph $G'_n$ with $n$ vertices and $21n$ edges, along with a proof that Maker can win the game playing on $E(G'_n)$. The graph $G'_n$ was constructed by dividing $n$ vertices into sets $V_0, V_1, \ldots , V_{m-1}, V'$, with $m \geq 40$, such that $d \leq |V_i| \leq d+1$, where $d = 38$, and $V' = \{ u_0, \ldots, u_{m-1} \}$ contains  exactly $m$ vertices. As for the edges of $G_n$, for every $0 \leq i <m$, every two vertices in $V_i$ are connected with an edge (i.e.~the graph induced on $V_i$ is a clique), and on top of that, for every $0 \leq i < m$, there is an edge between $u_i$ and every vertex of both $V_i$ and $V_{i+1}$ (indices are observed modulo $m$). All of this totals to no more than $(\frac{d}{2} + 2)n = 21n$ edges.

To prove that Maker can win Hamiltonicity game on $E(G'_n)$, the authors of~\cite{Sparse} relied on the fact that Maker wins $\ham_n$ on $E(K_n)$ for all $n \geq 38$, as $n=38$ was the smallest known such $n$ at the time. This was later improved by Hefetz and Stich~\cite{HefetzStich} to $n=29$, and now our Theorem~\ref{t:hcg} gives the optimal $n=8$. Having this in mind we can repeat the exact same proof with an adjusted construction of $G'_n$, reducing the size of all $V_i$ by setting $d=8$, and getting an immediate improvement of the upper bound to $\hat{m}(n) \leq 6n$.
	
But, it turns out that with the help of Fixed Hamiltonian Path game we are able to alter the construction and get an even better upper bound.

	\begin{proof}[Proof of Theorem \ref{t:app}]
		We will first describe our construction of the graph $G_n$ whose edge set will be the board of the game,
        and then we will prove that on this board Maker indeed can win Hamiltonicity game.

        Let $d = 7$, and let $m$ and $0 \leq r < d$ be integers such that $n = md + r$. The vertices of $G_n$ are partitioned into $m$ sets
        $V_0, V_1,\ldots , V_{m-1}$, such that $|V_i|=d+1$, $0 \leq i < r$, and $|V_i|=d$, $r \leq i < m$. Next,
		we arbitrarily pick vertices $a_i \in V_i$, $0 \leq i < m$, and define $W_i:= \{ a_{i-1} \} \cup V_i$, where indices are observed modulo $m$.
        The set of edges of $G_n$ is defined with
        $$ E(G_n) = \bigcup_{i=0}^{m-1} \Big\{ (v_1, v_2) \, \Big| \,\,\, v_1, v_2\in W_i, \,\, v_1 \not= v_2, \,\, \{ v_1, v_2\} \not= \{ a_{i-1}, a_i\} \Big\}. $$
        In other words, the graph $G_n$ consists of a ``cycle'' of cliques, where each two consecutive cliques overlap on a vertex, and each clique is missing
        exactly one edge (between the two vertices it shares with the neighboring cliques).
        The number of edges in $G_n$ is
        \begin{align*}
        E(G_n) &= r\binom{9}{2} + (m-r)\binom{8}{2} - m  \\
        & =  m\left(\binom{8}{2} -1 \right) + 8r \\
        & \leq \frac{n}{7} \left(4\cdot 7 - 1 \right) + 48.
        \end{align*}
        If $n \geq 336$, this is upper bounded by $4n$.

        To win Hamiltonicity game on $E(G_n)$, Maker simply plays $m$ Fixed Hamiltonian Path games in parallel, one on each of graphs $G_n[W_i]$, $0 \leq i < m$.
        More precisely, in each of those games Maker, as the second player, plays the game $\mathcal{FHP}_{|W_i|}$, with fixed vertices $a_{i-1}$ and $a_i$. As we previously observed, the edge $(a_{i-1}, a_i)$ does not participate in any winning sets, so Theorem~\ref{t:fhpg} guarantees Maker's win in all those games. Hence, for every $0 \leq i < m$, Maker can claim a Hamiltonian path on $G_n[W_i]$ with endpoints $a_{i-1}$ and $a_i$. These paths together form a cycle that covers all vertices of $G_n$, so Maker wins the Hamiltonicity game on $E(G_n)$.
	\end{proof}

Even though our proof works for $n \geq 336$, we can similarly obtain a non-trivial upper bound for any $n \geq 14$. Indeed, we can construct a graph consisting of a ``cycle'' of (two or more) cliques in the exact same way as in the previous proof, just making sure that each clique has at least $8$ vertices. Then we can borrow Maker's strategy from the previous proof, as well as the argument that he can claim a Hamiltonian cycle playing on that graph.

\section{Future work}
\label{s:fw}

	We have resolved three unbiased games played on the edges of a complete graph, Hamiltonicity game, Hamiltonian Path game and Fixed Hamiltonian Path game, for every $n$ and for each
    of the players starting the game. One related game that remains out of reach is the so-called \emph{Hamiltonian-Connectivity game}, where Maker's goal is to claim all edges of
    some spanning Hamiltonian-connected graph (a graph in which every two vertices are connected by a Hamiltonian path). Our approach was not efficent enough to handle Hamiltonian-Connectivity game since the nature of the winning sets is too complex making our pruning and winner detection procedures much less effective. This game may be utilized in a similar way as Fixed Hamiltonian Path game, but with
    more freedom -- what were \emph{two fixed (predetermined) vertices} in Fixed Hamiltonian Path game, become \emph{any two vertices} in Hamiltonian-Connectivity game. For that reason, we think that knowing the outcome of the game for every $n$ would be valuable.

    One consequence of finding the winner of a game using a computer is that typically we do not have a compact graph-theoretic description of a winning strategy for the winner. This is true for all games on small boards that we looked at -- we know who wins, but this knowledge is not accompanied by a coherent strategy (other than a huge list of winning moves, one for each game position).
	
	Because our computer program is implemented in a generic way it could be fairly easily adapted for
	attacking other positional games when played on small (enough) boards. We are hoping that other positional games can be solved using the tools developed in this paper.

\bibliographystyle{abbrv-pages}
\bibliography{references}

\end{document}